\documentclass[11pt]{amsart}
\usepackage{lmodern}
\usepackage{amsmath, amsthm, amssymb, amsfonts}
\usepackage[normalem]{ulem}
\usepackage{hyperref}

\usepackage{mathrsfs}

\usepackage{verbatim} 
\usepackage{longtable}

\usepackage{mathtools}

\usepackage{tikz}
\usetikzlibrary{decorations.pathmorphing}
\tikzset{snake it/.style={decorate, decoration=snake}}

\usepackage{caption}

\usepackage{tikz-cd}
\usetikzlibrary{arrows}

\theoremstyle{plain}
\newtheorem{thm}{Theorem}[section]
\newtheorem{cor}[thm]{Corollary}
\newtheorem{lem}[thm]{Lemma}
\newtheorem{prop}[thm]{Proposition}
\newtheorem{conj}[thm]{Conjecture}

\theoremstyle{definition}

\theoremstyle{remark}
\newtheorem{rmk}[thm]{Remark}

\newcommand{\BC}{{\mathbb{C}}}

\newcommand{\BP}{{\mathbb{P}}}
\newcommand{\BQ}{{\mathbb{Q}}}
\newcommand{\BR}{{\mathbb{R}}}

\newcommand{\BZ}{{\mathbb{Z}}}

\newcommand{\CA}{{\mathcal A}}
\newcommand{\CB}{{\mathcal B}}

\newcommand{\CD}{{\mathcal D}}
\newcommand{\CE}{{\mathcal E}}

\newcommand{\CU}{{\mathcal U}}

\newcommand{\CW}{{\mathcal W}}

\newcommand{\FM}{{\mathfrak{M}}}

\DeclareFontFamily{OT1}{rsfs}{}
\DeclareFontShape{OT1}{rsfs}{n}{it}{<-> rsfs10}{}
\DeclareMathAlphabet{\curly}{OT1}{rsfs}{n}{it}

\usepackage{tikz}
\usepackage{lmodern}
\usetikzlibrary{decorations.pathmorphing}

\begin{document}
\title[The $D$-equivalence conjecture]{The $D$-equivalence conjecture for hyper-K\"ahler varieties via hyperholomorphic bundles}
\date{\today}

\author[D. Maulik]{Davesh Maulik}
\address{Massachusetts Institute of Technology}
\email{maulik@mit.edu}

\author[J. Shen]{Junliang Shen}
\address{Yale University}
\email{junliang.shen@yale.edu}

\author[Q. Yin]{Qizheng Yin}
\address{Peking University}
\email{qizheng@math.pku.edu.cn}

\author[R. Zhang]{Ruxuan Zhang}
\address{Fudan University}
\email{rxzhang18@fudan.edu.cn}

\begin{abstract}
We show that birational hyper-K\"ahler varieties of $K3^{[n]}$-type are derived equivalent, establishing the $D$-equivalence conjecture in these cases. The Fourier--Mukai kernels of our derived equivalences are constructed from projectively hyperholomorphic bundles, following ideas of Markman. Our method also proves a stronger version of the $D$-equivalence conjecture for hyper-K\"ahler varieties of $K3^{[n]}$-type with Brauer classes.
\end{abstract}

\maketitle

\setcounter{tocdepth}{1} 

\tableofcontents
\setcounter{section}{-1}

\section{Introduction}

Throughout, we work over the complex numbers $\BC$. We recall that the $D$-equivalence conjecture \cite{BO, K} predicts that birational Calabi--Yau varieties have equivalent bounded derived categories of coherent sheaves.

\begin{conj}[$D$-equivalence conjecture]\label{conj}
If $X, X'$ are nonsingular projective birational Calabi--Yau varieties, then there is an equivalence of bounded derived categories
\[
D^b(X) \simeq D^b(X').
\]
\end{conj}


The purpose of this paper is to prove Conjecture \ref{conj} for hyper-K\"ahler varieties of $K3^{[n]}$-type; these are nonsingular projective varieties deformation equivalent to the Hilbert scheme of $n$ points on a $K3$ surface. More generally, our method reduces the $D$-equivalence conjecture for hyper-K\"ahler varieties to the construction of certain projectively hyperholomorphic bundles.

\begin{thm}\label{thm0.2}
Conjecture \ref{conj} holds for any hyper-K\"ahler varieties of $K3^{[n]}$-type.
\end{thm}

The $D$-equivalence conjecture has been proven by Bridgeland \cite{Bridgeland} for Calabi--Yau threefolds. For projective hyper-K\"ahler fourfolds, the $D$-equivalence conjecture holds by combining the classification results \cite{BHL, WW} and the case of Mukai flops by Kawamata \cite{K} and Namikawa \cite{Nam}. However, very few cases of this conjecture are known in dimension $>4$; see \cite{Ploog, ADM} for some partial results. Using equivalences obtained from window conditions, Halpern-Leistner~\cite{DHL} proved the $D$-equivalence conjecture for any hyper-K\"ahler variety which can be realized as a Bridgeland moduli space of stable objects on a (possibly twisted) $K3$ surface. Theorem \ref{thm0.2} generalizes Halpern-Leistner's result, but our construction of the derived equivalences is very different. We obtain explicit Fourier--Mukai kernels which rely on the theory of moduli spaces of hyper-K\"ahler manifolds and hyperholomorphic bundles \cite{Verb, Markman}; this is closer in spirit to the proposal of Huybrechts~\cite[Section 5.1]{H_ICM}. It would be interesting to find connections between the two approaches.

Our method in fact proves the following stronger, twisted version of the $D$-equivalence conjecture involving arbitrary Brauer classes. Let $X \dashrightarrow X'$ be a birational transform between hyper-K\"ahler varieties of ${K3}^{[n]}$-type. It naturally identifies the Brauer groups of $X,X'$: any Brauer class $\alpha \in \mathrm{Br}(X)$ induces a Brauer class $\alpha' \in \mathrm{Br}(X')$. 

\begin{thm}\label{thm0.3}
Let $X \dashrightarrow X'$ be as above, and let $\alpha$ be any Brauer class on $X$. Then there is an equivalence of bounded derived categories of twisted sheaves
\[
D^b(X, \alpha) \simeq D^b(X', \alpha').
\]    
\end{thm}

Theorem \ref{thm0.3} specializes to Theorem \ref{thm0.2} by taking $\alpha = 0$.

\subsection*{Acknowledgements}
We are grateful to Daniel Huybrechts, Zhiyuan Li, Eyal Markman, Alex Perry, and Ziyu Zhang for helpful discussions, and to the anonymous referee for comments and suggestions. D.M.~was supported by a Simons Investigator Grant. J.S.~was supported by the NSF Grant DMS-2301474 and a Sloan Research Fellowship. R.Z.~was supported by the NKRD Program of China No.~2020YFA0713200 and LNMS.

\section{Moduli of Hodge isometries}

Assume $n\geq 2$. We denote by $\Lambda$ the $K3^{[n]}$-lattice, which is isometric to $H^2(X, \BZ)$, equipped with the Beauville--Bogomolov--Fujiki (BBF) form, for any hyper-K\"ahler manifold $X$ of $K3^{[n]}$-type.\footnote{When we say that $X$ is a hyper-K\"ahler manifold or a manifold of $K3^{[n]}$-type, it means that $X$ is not necessarily projective.} In particular, we have a decomposition
\[
\Lambda = \Lambda_{K3} \oplus \BZ \delta, \quad \delta^2 = 2-2n
\]
with $\Lambda_{K3}$ the unimodular $K3$ lattice, so that any vector $\omega \in \Lambda$ can be expressed uniquely as 
\[
\omega = \omega_{K3} + \lambda \delta, \quad \omega_{K3} \in \Lambda_{K3}, \quad \lambda \in \BZ.
\]
A marking $(X, \eta_X)$ for a manifold $X$ of $K3^{[n]}$-type is an isometry $\eta_X: H^2(X, \BZ) \xrightarrow{\simeq} \Lambda$.

\subsection{Inseparable pairs}\label{sect1.1}
We denote by $\FM_\Lambda$ the moduli space of marked manifolds $(X,\eta_X)$ of~$K3^{[n]}$-type; it is naturally a non-Hausdorff complex manifold whose non-separation illustrates the complexity of the birational/bimeromorphic geometry of hyper-K\"ahler varieties/manifolds~\cite{H1}.

We say that a pair $(X,\eta_X), (X', \eta_{X'})$ is \emph{inseparable} if they represent inseparable points on the moduli space $\FM_{\Lambda}$; as a consequence of the global Torelli theorem, this is equivalent to the condition that $(X, \eta_X), (X', \eta_{X'})$ share the same period and lie in the same connected component of $\FM_{\Lambda}$.

Typical examples of inseparable pairs are given by bimeromorphic transforms. More precisely, a bimeromorphic map $X \dashrightarrow X'$ induces a natural identification $H^2(X, \BZ) = H^2(X', \BZ)$ respecting the Hodge structures. A marking~$\eta_X$ for $X$ then induces a marking $\eta_{X'}$ for $X'$, and the pair $(X, \eta_X), (X',\eta_{X'})$ is therefore inseparable. Note that inseparable points are not necessarily induced by bimeromorphic transforms \emph{directly}. As an example, we consider bimeromorphic $X,X'$ as above and assume that
\[
\rho: H^2(X', \BZ) \to H^2(X, \BZ)
\]
is a parallel transport respecting the Hodge structures. Then the pair 
\[
(X, \eta_X), \quad (X', \eta_{X'}), \quad \eta_{X'}:= \eta_X \circ \rho
\]
is inseparable. By \cite{H1} (see also \cite[Section 3.1]{Torelli}), every inseparable pair arises this way.

\subsection{Hodge isometries}

We recall the moduli space of Hodge isometries; this was used by Buskin~\cite{Buskin} and Markman \cite{Markman} to construct algebraic cycles realizing rational Hodge isometries. 

For $\phi \in O(\Lambda_\BQ)$, we define $\FM_{\phi}$ to be the moduli space of isomorphism classes of quadruples $(X, \eta_X, Y, \eta_Y)$ where $(X, \eta_X), (Y, \eta_Y) \in \mathfrak{M}_{\Lambda}$ are the corresponding markings and 
\[
\eta_Y^{-1}\circ \phi \circ \eta_X: H^2(X, \BQ) \to H^2(Y,\BQ)
\]
is a Hodge isometry sending some K\"ahler class of $X$ to a K\"ahler class of $Y$. We have the natural forgetful maps
\begin{gather*}
\Pi_1:  \FM_{\phi} \to \FM_\Lambda, \quad (X,\eta_X, Y, \eta_Y) \mapsto (X, \eta_X), \\
\Pi_2:  \FM_{\phi} \to \FM_\Lambda, \quad (X,\eta_X, Y, \eta_Y) \mapsto (Y, \eta_Y).
\end{gather*}
Any connected component $\FM^0_\phi $ of $\FM_\phi$ maps to a connected component of $\FM_\Lambda$ via $\Pi_i$ which  we denote by $\FM^0_\Lambda$.

\begin{lem}[{\cite[Lemma 5.7]{Markman}}]\label{lem1.1}
    The maps $\Pi_i: \FM^0_\phi \to \FM^0_\Lambda$ ($i=1,2$) between connected components are surjective.
\end{lem}


\begin{lem}\label{lem1.2}
Assume that the point $(X, \eta_X, Y,\eta_Y)$ lies in a connected component $\FM^0_\phi$. Assume further that $(X, \eta_X), (X',\eta_{X'})$ form an inseparable pair such that
\begin{equation}\label{point}
(X' ,\eta_{X'}, Y, \eta_Y) \in \FM_\phi.
\end{equation}
Then $(X' ,\eta_{X'}, Y, \eta_Y)$ lies in the same component $\FM^0_\phi$.
\end{lem}

Note that (\ref{point}) is equivalent to the condition that $\eta_Y^{-1} \circ \phi \circ \eta_{X'}$ sends some K\"ahler class of~$X'$ to a K\"ahler class of $Y$. 

\begin{proof}
Both $(X, \eta_X), (X',\eta_{X'})$ lie in the same connected component of $\FM_\Lambda$ which we call $\mathfrak{M}^0_\Lambda$. We first find paths in $\FM^0_\Lambda$ connecting both points to $(X_0, \eta_{X_0}) \in \FM^0_\Lambda$ with $\mathrm{Pic}(X_0) =0$. By Lemma \ref{lem1.1}, we can lift these paths to $\FM_\phi$, which connect $(X, \eta_X, Y,\eta_Y)$ to $(X_0, \eta_{X_0}, Y_0, \eta_{Y_0})$, and $(X' ,\eta_{X'}, Y, \eta_Y)$ to $(X_0, \eta_{X_0}, Y'_0, \eta_{Y'_0})$ respectively. On one hand, by considering the projection $\Pi_2$, we know that the two points $(Y_0,\eta_{Y_0}), (Y'_0,\eta_{Y'_0})$ lie in the same connected component of $\FM_\Lambda$; on the other hand, the Hodge isometry condition ensures that both of them share the same period \cite[Lemma 5.4]{Markman} and they have trivial Picard group. By the global Torelli theorem, we must have $(Y_0, \eta_{Y_0}) = (Y'_0 ,\eta_{Y'_0})$. This completes the proof.
\end{proof}

Suppose we are given a point $(X, \eta_X, Y, \eta_Y)$ in $\FM_\phi$, and K\"ahler classes $\omega_X, \omega_Y$ on~$X, Y$ which are identified via $\eta_Y^{-1}\circ \phi \circ \eta_X$.
Using this data, one can define a \emph{diagonal twistor line}~$\ell \subset {\FM}_\phi$ which lifts
the twistor lines associated to $(X, \omega_X)$ and $(Y, \omega_Y)$
on $\mathfrak{M}_\Lambda$. A \emph{generic diagonal twistor path} 
on $\FM_\phi$ is given by a chain of diagonal twistor lines such that, at each node of the chain, the associated hyper-K\"ahler manifolds have trivial Picard group. Generic diagonal twistor paths are used in Theorem \ref{thm1.4} below to deform certain Fourier--Mukai kernels.

\subsection{Brauer groups}

Assume that $X$ is a manifold of $K3^{[n]}$-type. Since $X$ has no odd cohomology, the discussion in \cite[Section 4.1]{Ca} yields the following explicit description of the (cohomological) Brauer group:
\begin{equation}\label{Br}
\mathrm{Br}(X) = \left( H^2(X, \BZ) / \mathrm{Pic}(X)\right) \otimes \BQ/\BZ.
\end{equation}
In particular, given a bimeromorphic map $X \dashrightarrow X'$ between manifolds of $K3^{[n]}$-type, there is a natural identification 
\[
\mathrm{Br}(X) = \mathrm{Br}(X')
\]
since both $H^2(-,\BZ)$ and $\mathrm{Pic}(-)$ are identified for $X$ and $X'$. The description (\ref{Br}) also allows us to present a Brauer class in the form
\begin{equation}\label{B_Field}
\left[ \frac{\beta}{d}\right] \in \mathrm{Br}(X), \quad \beta\in H^2(X,\BZ), \quad d\in \BZ_{> 0};
\end{equation}
this is referred to as the ``$B$-field''.


We note that the cohomology $H^2(X, \BZ)$ forms a trivial local system over any connected component of the moduli space $\FM^0_{\Lambda}$; therefore (\ref{B_Field}) for a single $X$ presents a Brauer class for any point in the component $\FM^0_{\Lambda}$ containing $(X, \eta_X)$. 

\subsection{Projectively hyperholomorphic bundles}\label{Section1.3}
Using the Bridgeland--King--Reid (BKR) correspondence \cite{BKR}, Markman constructed in \cite{Markman} a class of projectively hyperholomorphic bundles which we recall here. We consider a projective $K3$ surface $S$ with $\mathrm{Pic}(S)=\BZ H$. Assume that $r,s$ are two coprime integers with $r\geq 2$. Assume further that the Mukai vector
\[
v_0:=(r, mH, s) \in H^*(S, \BZ)
\]
is isotropic, \emph{i.e.}~$v_0^2=0$, and that all stable sheaves on $S$ with Mukai vector $v_0$ are stable vector bundles.\footnote{In \cite{Markman}, Markman only considered the case $m=1$; here considering large $\pm m$ is crucial for our purpose. Using \cite[Proposition 2.2]{H2} (see also \cite[Theorem 2.2]{Y}), Markman's argument works identically in this generality.} Let $M$ be the moduli space of such stable vector bundles. Then $M$ is again a $K3$ surface, and the coprime condition of $r, s$ ensures the existence of a universal rank $r$ bundle $\CU$ on $M \times S$. Conjugating the BKR correspondence, we obtain a vector bundle $\CU^{[n]}$ on $M^{[n]} \times S^{[n]}$ of rank
\[
\mathrm{rk}(\CU^{[n]}) = n!r^n;
\]
see \cite[Lemma 7.1]{Markman}. This vector bundle induces a derived equivalence
\begin{equation}\label{Deq}
\Phi_{\CU^{[n]}}: D^b(M^{[n]}) \xrightarrow{\simeq} D^b(S^{[n]}).
\end{equation}
Markman further showed in \cite[Section 5.6]{Markman} that the characteristic class of $\CU^{[n]}$ induces a Hodge isometry
\[
\phi_{\CU^{[n]}}: H^2(M^{[n]}, \BQ) \to H^2(S^{[n]}, \BQ). 
\]
Under the natural identification
\begin{equation}\label{Hilb}
H^2(M^{[n]}, \BQ) = H^2(M, \BQ)\oplus \BQ \delta, \quad H^2(S^{[n]}, \BQ) = H^2(S, \BQ)\oplus \BQ  \delta,
\end{equation}
this Hodge isometry is of the form
\[
\phi_{\CU^{[n]}} =  (\phi_\CU, \mathrm{id}), \quad \phi_\CU : H^2(M, \BQ) \to H^2(S, \BQ),
\]
where $\phi_\CU$ is the Hodge isometry of $K3$ surfaces induced by $\CU$; see \cite[Corollary 7.3]{Markman}.

The key results, which are summarized in the following theorem, show that the Fourier--Mukai kernel $\CU^{[n]}$, as a projectively hyperholomorphic bundle, deforms along generic diagonal twistor paths. Moreover, at each point of the path, it induces a (twisted) derived equivalence:

\begin{thm}[\cite{Markman, Polish}]\label{thm1.4}
There exist markings $\eta_{M^{[n]}} ,\eta_{S^{[n]}}$ for the Hilbert schemes $M^{[n]}, S^{[n]}$ respectively, which induce $\phi \in O(\Lambda_\mathbb{Q})$ via $\phi_{\CU^{[n]}}$, such that the connected component containing the quadruple 
\[
(M^{[n]}, \eta_{M^{[n]}}, S^{[n]}, \eta_{S^{[n]}}) \in \FM^0_{\phi}
\]
satisfies the following:
\begin{enumerate}
    \item[(a)] For every point $(X, \eta_X, Y, \eta_Y)$ lying in the component $\FM^0_{\phi}$, there exists a twisted vector bundle $(\CE, \alpha_\CE)$ on $X \times Y$, which is deformed from $\CU^{[n]}$ along a generic diagonal twistor path.
    \item[(b)] Using the form (\ref{B_Field}), the Brauer class in (a) is presented by 
    \[
    \alpha_\CE = \left[- \frac{c_1(\CU^{[n]})}{\mathrm{rk}(\CU^{[n]})} \right]. 
    \]
    Here we view $H^2(X\times Y, \BZ) = H^2(X, \BZ) \oplus H^2(Y, \BZ)$ as a trivial local system over the moduli space $\FM^0_{\phi}$ via the markings.
    \item[(c)] Further assume that $X, Y$ are varieties. Then the twisted bundle $(\CE, \alpha_\CE)$ induces an equivalence of twisted derived categories
    \[
    \Phi_{(\CE, \alpha_\CE)}: D^b(X, \alpha_X) \xrightarrow{\simeq} D^b(Y, \alpha_Y), \quad \alpha_X = \left[ \frac{a_X}{\mathrm{rk}(\CE)}\right], \quad \alpha_Y = \left[-\frac{a_Y}{\mathrm{rk}(\CE)} \right],
    \]
    where $a_X \in H^2(X, \BZ), a_Y \in H^2(Y, \BZ)$ are given by 
    \[
    c_1(\CU^{[n]})  = a_X + a_Y \in H^2(X, \BZ) \oplus H^2(Y, \BZ).
    \]
\end{enumerate}
\end{thm}

\begin{proof}
(a) was proven in \cite[Theorem 8.4]{Markman}; Markman showed that $\CU^{[n]}$ on $M^{[n]}\times S^{[n]}$ is projectively slope-stable hyperholomorphic in the sense of \cite{Verb, BB} which allows him to deform it along diagonal twistor paths to all points in the component $\FM^0_{\phi}$.


(b) can be obtained by applying C\u{a}ld\u{a}raru's result \cite[Theorem 4.1]{Ca} along the diagonal twistor paths; see the discussion in \cite[Section 2.3]{Polish}.

(c) was proven in \cite[Theorem 2.3]{Polish}. More precisely, the condition that a twisted bundle induces a twisted derived equivalence can be characterized by cohomological properties \cite[Theorem~3.2.1]{Ca0}. These properties are preserved along a twistor path due to the fact that the cohomology of slope-polystable hyperholomorphic bundles is invariant under hyper-K\"ahler rotations \cite[Corollary 8.1]{Verb2}. Therefore we ultimately reduce the desired cohomological properties to those for $M^{[n]} \times S^{[n]}$ which are given by the original equivalence (\ref{Deq}).
\end{proof}

\subsection{Birational geometry and MBM classes}
The birational geometry of hyper-K\"ahler varieties is governed by certain integral \emph{primitive} cohomology classes, called the monodromy birationally minimal (MBM) classes. We refer to \cite{AVsurvey} for an introduction to these classes. In the following, we summarize some results which are needed in our proof.

Let $X$ be a variety of $K3^{[n]}$-type. We consider its birational K\"ahler cone $\mathcal{BK}_X$ and the positive cone $\mathcal{C}_X$:
\[
\mathcal{BK}_X \subset \mathcal{C}_X  \subset H^{1,1}(X, \BR).
\]
The positive cone is convex and admits a wall-and-chamber structure. The closure of the birational K\"ahler cone within the positive cone is a convex sub-cone \cite{H_cone}, which inherits a wall-and-chamber structure. Furthermore, by a result of Amerik--Verbitsky \cite{AV}, and independently Mongardi \cite{Mon}, all the walls are governed by the MBM classes.

\begin{thm}[\cite{Mon, AV}]\label{thm1.5}
    Any wall of $\mathcal{C}_X$ is described by a hyperplane of the form 
    \[
    \CW^{\perp}: = \{\omega \in H^{1,1}(X,\BR), \, (\omega, \CW) = 0\}\subset H^{1,1}(X, \BR)
    \]
    with $\CW$ an algebraic MBM class in $\mathrm{Pic}(X)$. Here the pairing is with respect to the BBF form. Moreover, any chamber in $\mathcal{C}_X$ can be realized as the K\"ahler cone of a birational hyper-K\"ahler~$X'$ through a parallel transport $\rho: H^2(X', \BZ) \to H^2(X, \BZ)$ respecting the Hodge structures.
\end{thm}

Note that any chamber in $\mathcal{BK}_X \subset \mathcal{C}_X$ is given by the pullback of the K\"ahler cone via a birational transform $X \dashrightarrow X'$ of hyper-K\"ahler varieties. By the discussions of Section \ref{sect1.1}, any chamber of $\mathcal{C}_X$ corresponds to a marked variety $(X' ,\eta_{X'})$ of $K3^{[n]}$-type such that the pair $(X, \eta_X), (X', \eta_{X'})$ is inseparable.

We also need the following boundedness result, which notably implies that wall-and-chamber structure of $\mathcal{C}_X$ is locally polyhedral; see \cite[Remark~8.2.3]{HK3} for a proof of the implication. The boundedness result was essentially obtained by \cite{BHT}, as explained in~\cite[Section 6.2]{AV}.

\begin{thm}[\cite{BHT, AV}]\label{thm1.6}
There is a constant $C_0 >0$, such that for any variety $X$ of $K3^{[n]}$-type and any MBM class $\CW \in H^2(X, \BZ)$ we have
\[
0 < -\CW^2 < C_0. 
\]
Here the norm is with respect to the BBF form.  
\end{thm}

For any rational Hodge isometry $\phi: H^2(X,\BQ) \to H^2(Y,\BQ)$ between varieties of~$K3^{[n]}$-type, which sends an MBM class $\CW_X$ on $X$ to a class proportional to an MBM class~$\CW_Y$ on~$Y$, there exist coprime integers $a,b$ such that
\[
\phi(\CW_X) = \frac{a}{b} \CW_Y.
\]
The following is an immediate consequence of Theorem \ref{thm1.6}.

\begin{cor}\label{cor1.7}
    For any $X,Y, \phi, \CW_X, \CW_Y$ as above, we have
    \[
    a^2 < C_0, \quad b^2 < C_0.
    \]
\end{cor}

\begin{proof}
    Since $\phi$ is an isometry, we have
    \[
    \frac{a^2}{b^2} = \frac{\CW_X^2}{\CW_Y^2}
    \]
    By Theorem \ref{thm1.6}, both $-\CW_X^2$ and $-\CW_Y^2$ are positive integers $< C_0$. The corollary follows from the assumption that $a,b$ are coprime.
\end{proof}

\subsection{Proof strategy}\label{Sec1.5}
We discuss the strategy of the proof of Theorem \ref{thm0.3}; Theorem \ref{thm0.2} is then deduced as a special case. 

Let $X$ be a variety of $K3^{[n]}$-type. It suffices to prove Theorem \ref{thm0.3} for a hyper-K\"ahler birational model $X'$ with a birational map $X \dashrightarrow X'$ which corresponds to a chamber in $\mathcal{BK}_X$ adjacent to the K\"ahler cone of~$X$. By Theorem \ref{thm1.6}, the wall between these two chambers is given by an algebraic MBM class $\CW \in \mathrm{Pic}(X)$.

Now we choose a $K3$ surface $S$ and a Mukai vector $v_0=(r,mH,s)$ as in the beginning of Section \ref{Section1.3}, which yields the Hodge isometry $\phi_{\CU^{[n]}}$. Associated to these, we have the moduli space of Hodge isometries~$\FM_{\phi}$, and the component $\FM^0_\phi$ that contains the quadruple $(M^{[n]}, \eta_{M^{[n]}}, S^{[n]}, \eta_{S^{[n]}})$. 

For the given birational $X,X'$, by Lemma \ref{lem1.1}, we can complete them to a pair of quadruples
\[
(Y, \eta_{Y}, X, \eta_X),\quad (Y', \eta_{Y'}, X', \eta_{X'}) \in \FM^0_{\phi}
\]
such that the marking $\eta_{X'}$ is induced by $\eta_X$ via the birational map $X \dashrightarrow X'$.\footnote{Here we would like $X, X'$ to be deformed from $S^{[n]}$ later in Section \ref{section2}.} In particular, the pair $(X, \eta_X), (X', \eta_{X'})$ is inseparable. We note that the pair $(Y,\eta_Y), (Y', \eta_{Y'})$ is also inseparable. This is because they share the same period and lie in the same connected component of $\FM_\Lambda$. Moreover, by definition, $\phi^{-1}$ sends a K\"ahler class of $X$ (resp.~$X'$) to a K\"ahler class of~$Y$ (resp.~$Y'$).\footnote{Here we suppress the markings and still use $\phi$ to denote the Hodge isometry $H^2(Y,\BQ) \to H^2(X,\BQ)$ for notational convenience.} Therefore, if 
\begin{equation}\label{wall_condition}
\textup{$\phi^{-1}$ does not send $\CW$ to a class on $Y$ that is proportional to an MBM class}, 
\end{equation}
there must be a point on the wall separating the K\"ahler cones of $X,X'$ which is sent to the interior of a chamber of the positive cone $\mathcal{C}_{Y}$. In particular, there exists a hyper-K\"ahler birational model ${Y''}$ of $Y$ with a marking $({Y''}, \eta_{Y''})$ such that the pair $(Y,\eta_Y), (Y'', \eta_{Y''})$ is inseparable and 
\[
(Y'', \eta_{Y''}, X, \eta_X),\quad (Y'', \eta_{Y''}, X', \eta_{X'}) \in \FM_\phi.
\]
Furthermore, by Lemma \ref{lem1.2}, both points lie in the connected component we started with:
\[
(Y'', \eta_{Y''}, X, \eta_X),\quad (Y'', \eta_{Y''}, X', \eta_{X'}) \in \FM^0_\phi.
\]

By Theorem \ref{thm1.4}(b, c), we obtain Brauer classes $\alpha_X, \alpha_{Y''}$ on $X,Y''$ respectively, such that
\begin{equation*}\label{1.4_1}
 D^b(Y'', \alpha_{Y''}) \simeq D^b(X, \alpha_X), \quad   D^b(Y'', \alpha_{Y''}) \simeq D^b(X', \alpha_{X'}).
 \end{equation*}
Here the Brauer classes $\alpha_X, \alpha_{Y''}$ only depend on the markings $(X, \eta_X), (Y'',\eta_{Y''})$ respectively, and the Brauer class $\alpha_{X'}$ is induced by $\alpha_{X}$. Combining both equivalences yields 
\begin{equation*}\label{tw_D_eq}
D^b(X, \alpha_X) \simeq D^b(X', \alpha_{X'})
\end{equation*}
whose Fourier--Mukai kernel is the composition of two (twisted) hyperholomorphic bundles.

In the next section, we show that for any pair $X, X'$ as above with a Brauer class $\alpha \in \mathrm{Br}(X)$ and an algebraic MBM class~$\CW \in \mathrm{Pic}(X)$, a careful choice of the $K3$ surface $S$ and the Mukai vector $v_0=(r, mH,s)$ as in Section~\ref{Section1.3} can simultaneously ensure that the condition (\ref{wall_condition}) holds and the induced Brauer class is as desired:
\begin{equation}\label{vanish2}
\alpha_X = \alpha.
\end{equation}
This completes the proof of Theorem \ref{thm0.3}.

\begin{rmk}
For a general birational transform $X \dashrightarrow X'$ of varieties of $K3^{[n]}$-type, which does not correspond to adjacent chambers in the birational K\"ahler cone $\mathcal{BK}_X$, our proof realizes the derived equivalence
\[
D^b(X,\alpha) \simeq D^b(X', \alpha')
\]
via two sequences of varieties $X_1, \ldots, X_{t - 1}$ and $Y_1, \ldots, Y_t$, with each $X_i$ birational to $X, X'$, such that
\begin{multline} \label{eq:seq}
D^b(X, \alpha) \simeq D^b(Y_1, \alpha_{Y_1}) \simeq D^b(X_1, \alpha_{X_1}) \simeq D^b(Y_2, \alpha_{Y_2}) \simeq \cdots \\
\simeq D^b(Y_{t - 1}, \alpha_{Y_{t - 1}}) \simeq D^b(X_{t - 1}, \alpha_{X_{t-1}}) \simeq D^b(Y_t, \alpha_{Y_t}) \simeq D^b(X', \alpha').
\end{multline}
Each of the derived equivalences in \eqref{eq:seq} is induced by a (twisted) hyperholomorphic bundle.
\end{rmk}

\section{Proof of Theorem \ref{thm0.3}}\label{section2}

From now on, we fix a variety $X$ of $K3^{[n]}$-type, a Brauer class $\alpha \in \mathrm{Br}(X)$, and an algebraic MBM class $\CW \in \mathrm{Pic}(X)$ as in Section \ref{Sec1.5}. In particular, the variety $X$ has Picard rank~$\geq 2$.\footnote{Theorem \ref{thm0.3} is automatically true if $X$ has Picard rank $1$, since any birational transform $X \dashrightarrow X'$ is necessarily an isomorphism.} Using \eqref{Br} and \eqref{B_Field}, we present the Brauer class $\alpha$ by a class in the rational transcendental lattice~$T(X)_\BQ \subset H^2(X, \BQ)$:
\[
\alpha = \left[-\frac{\CB}{d}\right], \quad \CB \in T(X), \quad d \in \BZ_{> 0}.
\]
Up to adjusting $-\frac{\CB}{d}$ by an integral class in $T(X) \subset H^2(X, \BZ)$, we may further assume that the class $\CB$ satisfies
\[
\CB^2 = 2e > 0.
\]

\subsection{Divisor classes}
Recall that the divisibility $\mathrm{div}(\omega)$ of a class $\omega \in H^2(X, \BZ)$ is the positive generator of the subgroup
\[
\{(\omega, \mu), \, \mu \in H^2(X, \BZ)\} \subset \BZ
\]

\begin{lem}
There exists a class $\CA\in \mathrm{Pic}(X)$ such that
\[
(\CA, \CW) \neq 0, \quad \mathrm{div}(\CA) = 1.
\]
\end{lem}

\begin{proof}
    We pick a marking identifying $H^2(X, \BZ)$ with a $K3^{[n]}$-lattice $\Lambda = \Lambda_{K3} \oplus \BZ\delta$. 
    For any~$g\in O(\Lambda)$, since $g(\delta)^\perp$ is a unimodular $K3$-lattice, any primitive vector $\omega \in g(\delta)^\perp \subset \Lambda$ satisfies $\mathrm{div}(\omega) =1.$    
   We would like to choose $g$ so that there exists 
   $\CA \in g(\delta)^\perp \cap \mathrm{Pic}(X)$ satisfying $(\CA, \CW) \neq 0$.     
   In other words, we want 
   \[
   g(\delta)^\perp \cap \mathrm{Pic}(X) \neq \CW^\perp \cap  \mathrm{Pic}(X).
   \]
   
   If we base change to $\mathbb{C}$, the set of $g \in O(\Lambda)_{\mathbb{C}}$ such that
   \[
   g(\delta)^\perp \cap \mathrm{Pic}(X)_{\mathbb{C}} \neq \CW^\perp \cap  \mathrm{Pic}(X)_{\mathbb{C}}
   \]
    is open in the Zariski topology.  Furthermore, it is nonempty since $X$ has Picard rank $\geq 2$. Since $O(\Lambda)$ is Zariski-dense in $O(\Lambda)_{\mathbb{C}}$, we can find $g \in O(\Lambda)$ satisfying this condition as well.
\end{proof}

Up to replacing $\CA$ by $-\CA$, we may assume
\[
C_1:= (\CA, \CW) > 0
\]
which we fix from now on.

\begin{prop}\label{prop2.2}
For any $N>0$, there exists a class $\CD\in \mathrm{Pic}(X)$ of divisibility $1$, satisfying
\[
\CD^2>N, \quad (\CD,\CW) = C_1.
\]
\end{prop}

\begin{proof}
    Since $X$ has Picard rank $\geq 2$, we have $\CW^\perp \cap \mathrm{Pic}(X) \neq 0$. Pick an element 
    \[
    \omega \in \CW^\perp \cap \mathrm{Pic}(X), \quad \omega^2 >0.
    \]
    Then for large enough $t \in \BZ_{>0}$, we have
    \[
    (\CA + t\omega, \CW) = C_1, \quad (\CA + t\omega)^2 >N.
    \]
    It suffices to show that there exist infinitely many choices of $t\in \BZ_{>0}$ satisfying 
    \[
    \mathrm{div}(\CA + t\omega)=1.
    \]
    We pick an integral class $\mu \in H^2(X,\BZ)$ such that
    \[
    (\CA, \mu)=1, \quad (\omega, \mu) \neq 0;
    \]
    then we pick another integral class $\nu \in H^2(X, \BZ)$ such that
    \[(\CA, \nu) = 0, \quad (\omega, \nu) \neq 0.
    \]
    By Dirichlet's theorem on primes in arithmetic progressions, there exists a sufficiently large $t\in \BZ_{>0}$ such that the absolute value of $1 + t(\omega, \mu)$ is a prime number. We claim that for such a choice of $t$, the class $\CA + t\omega$ must have divisibility $1$. This follows immediately from the observation that 
    \[
    \mathrm{div}(\CA + t\omega) \mid 1 + t(\omega, \mu), \quad \mathrm{div}(\CA + t\omega) \mid (\omega, \nu). \qedhere
    \]
\end{proof}

\subsection{Mukai vectors}
We construct the $K3$ surface $S$ and the Mukai vector $v_0$ of Section \ref{Sec1.5}.

By Proposition \ref{prop2.2}, we can find $\CD \in \mathrm{Pic}(X)$ with 
\begin{equation}\label{choice}
\mathrm{div}(\CD)=1, \quad (\CD,\CW) = C_1 > 0, \quad \CD^2 = 2g > 2C_0C_1,
\end{equation}
where $C_0$ is the constant in Theorem \ref{thm1.6}. Repeating the same argument as in Proposition \ref{prop2.2}, we also find $t \in \BZ_{>0}$ such that
\[
\mathrm{div}(\CD+4gtd \CB) =1.
\]

Let $(S, H)$ be a primitively polarized $K3$ surface of Picard rank $1$ of degree
\[
H^2 = 2g\left(1+ 4gt^2d^4(n-1) +16gt^2d^2e\right) > 0.
\]
We observe that both classes
\[
H - 2gtd^2\delta\in \mathrm H^2(S^{[n]}, \BZ), \quad \CD+ 4gtd \CB \in H^2(X,\BZ)
\]
are of divisibility 1 and have the same norm, where we have used that $(\CD, \CB)=0$ since $\CB$ is transcendental. Therefore, by \cite[Example 3.8]{GHS} and \cite[Theorem 9.8]{Torelli}, there is a parallel transport
    \[
    \rho: H^2(S^{[n]}, \BZ)  \to H^2(X, \BZ)
    \]
satisfying
\begin{equation} \label{eq:rhorho}
\rho(H - 2gtd^2\delta) = \epsilon (\CD+ 4gtd \CB),
\end{equation}
where $\epsilon = \pm 1$ is a sign determined by the orientation. 

We now consider the Mukai vector
\[
v_0:= \left(16gt^2d^4, \, \epsilon\cdot  4td^2H, \, 1+ 4gt^2d^4(n-1) +16gt^2d^2e\right),
\]
which clearly satisfies
\[
\mathrm{gcd}\left(16gt^2d^4, 1+ 4gt^2d^4(n-1) +16gt^2d^2e\right)=1, \quad 16gt^2d^4 \geq 2, \quad v_0^2 = 0.
\]
Since $g > 1$, we have
\[
4gtd^2 \nmid \frac{H^2}{2} + 1 = g\left(1+ 4gt^2d^4(n-1) +16gt^2d^2e\right) + 1
\]
which by \cite[Lemma 1.2]{Y} implies that all stable sheaves on $S$ with Mukai vector $v_0$ are stable vector bundles. Therefore, the moduli space $M$ of stable vector bundles on $S$ with Mukai vector $v_0$ is a $K3$ surface of Picard rank $1$ with a universal bundle $\CU$ on $M \times S$ which we fix from now on. Also fixed are the markings $\eta_{M^{[n]}}, \eta_{S^{[n]}}$ as in Theorem \ref{thm1.4}, as well as the induced marking
\[
\eta_X := \eta_{S^{[n]}} \circ \rho^{-1}: H^2(X, \BZ) \xrightarrow{\simeq} \Lambda.
\]


\begin{prop}\label{prop2.3} Let $S, M, \CU$ be as above.
\begin{enumerate}
    \item[(a)] The primitive polarization $\widehat{H}$ of $M$ satisfies $\widehat{H}^2 = H^2$.
    \item[(b)] Let $s\in S$ be a point. Assume that the vector bundle $\CU|_s$ has Mukai vector 
    \[
    \widehat{v}_0 = (16gt^2d^4, k\widehat{H}, \widehat{s}) \in H^*(M, \BZ). 
    \]
    Then we have
    \[
    \mathrm{gcd}( 16gt^2d^4, k ) = 4td^2.
    \]
\end{enumerate}
\end{prop}

\begin{proof}
    (a) follows from \cite[Appendix A]{Y}. For (b), we note that \cite[Theorem 2.2]{Y} implies that the Mukai vector $\widehat{v}_0$ is primitive with $\widehat{v}_0^2=0$. Using (a), we deduce that
    \[
    \widehat{s} = \left(\frac{k}{4td^2} \right)^2 \left(1+ 4gt^2d^4(n-1) +16gt^2d^2e \right) \in \BZ.
    \]
    Therefore, we have that $k$ is divisible by $4td^2$, which shows 
    \[
    4td^2 \mid \mathrm{gcd}(16gt^2d^4, k).
    \]
    On the other hand, if $\frac{k}{4td^2}$ is not coprime to $16gt^2d^4$, the Mukai vector $\widehat{v}_0$ is divisible by their common factor. This contradicts the fact that $\widehat{v}_0$ is primitive.
\end{proof}

\subsection{End of proof} We complete the proof using the $K3$ surface $S$, the Mukai vector $v_0$, and the universal bundle $\CU$ constructed in the last section. This gives the vector bundle $\CU^{[n]}$ on~$M^{[n]} \times S^{[n]}$. We write 
\[
c_1(\CU^{[n]}) = a_{M^{[n]}} + a_{S^{[n]}} \in H^2(M^{[n]}, \BZ) \oplus H^2(S^{[n]}, \BZ)
\]
with
\[
a_{M^{[n]}} \in H^2(M^{[n]}, \BZ), \quad a_{S^{[n]}}\in H^2(S^{[n]}, \BZ).
\]
Recall the natural identification
\begin{equation}\label{H2}
H^2(S^{[n]}, \BZ) = H^2(S, \BZ) \oplus \BZ\delta.
\end{equation}
By \cite[Equation (7.11)]{Markman}, we can present the class $a_{S^{[n]}}$ using (\ref{H2}):
\begin{equation*}
a_{S^{[n]}} = \mathrm{rk}(\CU^{[n]}) \cdot \left( \frac{\epsilon \cdot 4td^2H}{16gt^2d^4} -\frac{\delta}{2} \right) \in H^2(S^{[n]}, \BZ).
\end{equation*}
Via the parallel transport $\rho$ and (\ref{eq:rhorho}), we obtain
\begin{align*}
\rho\left(\frac{a_{S^{[n]}}}{\mathrm{rk}(\CU^{[n]})} \right) & = \rho\left(\frac{\epsilon \cdot H}{4gtd^2} -\frac{\delta}{2} \right)\\
& = \rho\left(\epsilon \left(\frac{H}{4gtd^2} -\frac{\delta}{2} \right)  + (\epsilon - 1) \frac{\delta}{2} \right) \\
& = \frac{\epsilon \cdot \rho(H - 2gtd^2\delta)}{4gtd^2} + \frac{(\epsilon - 1)}{2}\rho(\delta) \\
& = \frac{\CD + 4gtd\CB}{4gtd^2} + \frac{(\epsilon - 1)}{2}\rho(\delta) \\
& = \frac{\CB}{d} + \textup{[class in $\mathrm{Pic}(X)_\BQ$]} +\textup{[class in $H^2(X, \BZ)$]} \in H^2(X, \BQ).
\end{align*}
Hence, by Theorem \ref{thm1.4}(c), we have
\[
\alpha_X = \left[ -\rho\left(\frac{a_{S^{[n]}}}{\mathrm{rk}(\CU^{[n]})} \right)\right] = \left[-\frac{\CB}{d}\right] = \alpha.
\]

To complete the proof, it remains to address (\ref{wall_condition}). This is given by the following proposition.

\begin{prop}
Let $\FM_\phi^0$ be the connected component of the moduli space of Hodge isometries constructed from $S,M,\CU$ as above. For any quadruple 
\[
(Y, \eta_{Y}, X, \eta_X) \in \FM_{\phi}^0
\]
with $X, \CW$ fixed as above, the class $\phi^{-1}(\CW) \in H^2(Y, \BQ)$ is not proportional to any MBM class on $Y$. Here we suppress the markings and still use $\phi$ to denote the Hodge isometry $H^2(Y,\BQ) \to H^2(X,\BQ)$ for notational convenience.
\end{prop}

\begin{proof}

The main idea of the argument is that, for our choice of the Mukai vector $v_0$, 
by a calculation of Buskin \cite{Buskin},
the rational Hodge isometry $\phi^{-1}$ is conjugate to a reflection by a vector of large norm.  By Corollary \ref{cor1.7}, we then show that it cannot send $\CW$ to a class proportional to an MBM class.

The details are as follows. Since the MBM classes are deformation invariant, we only need to treat the Hodge isometry
\[
\phi_{\CU^{[n]}}: H^2(M^{[n]}, \BQ) \to H^2(S^{[n]}, \BQ)
\]
which can be further simplified under the identification (\ref{Hilb}):
\[
(\phi_\CU, \mathrm{id}): H^2(M, \BQ) \oplus \BQ\delta \to H^2(S, \BQ) \oplus \BQ\delta.
\]
Assume that 
\begin{equation}\label{bad}
\phi_{\CU^{[n]}}^{-1}(\rho^{-1}(\CW)) = \frac{b}{a} \CW'
\end{equation}
with $\CW'$ an MBM class on $M^{[n]}$ and $a, b$ coprime. We write 
\[
\rho^{-1}(\CW) = \CW_{K3} + \lambda \delta, \quad \CW_{K3} \in H^2(S,\BZ), \quad \lambda\in \BZ.
\]
The equation (\ref{bad}) implies that $\phi_{\CU}^{-1}(a\CW_{K3})$ is an integral class. By the formula right before \cite[Conclusion 3.8]{Buskin}, the integrality forces the pairing
\[
(H, a\CW_{K3}) \in \BZ
\]
to be divisible by
\[
\frac{16gt^2d^4}{\mathrm{gcd}\left(16gt^2d^4, 4td^2k\right)} = g,
\]
where we have used Proposition \ref{prop2.3}(b) in the last equation.

On the other hand, we have
\[
(H, a\CW_{K3}) = (H, \rho^{-1}(a\CW)) = (\rho(H), a\CW) = \epsilon  (\CD,a\CW) + \textup{[integer divisible by $g$]},
\]
where the last equality uses \eqref{eq:rhorho}. In particular, we find
\[
g \mid (\CD, a\CW) = aC_1. 
\]
Combined with Corollary \ref{cor1.7}, this implies 
\[
g \leq a^2C_1 < C_0C_1
\]
which contradicts our choice of $\CD$ in (\ref{choice}). This shows that (\ref{bad}) cannot hold, which proves the proposition.
\end{proof}

In conclusion, both (\ref{wall_condition}) and (\ref{vanish2}) are settled by our choice of the $K3$ surface $S$ and the Mukai vector $v_0$; the proof of Theorem \ref{thm0.3} is now complete. \qed


\begin{thebibliography}{99}

\bibitem{ADM} N. Addington, W. Donovan, and C. Meachan, {\em Moduli spaces of torsion sheaves on $K3$ surfaces and derived equivalences,} J. Lond. Math. Soc. (2) 93 (2016), no. 3, 846--865.


\bibitem{AV} E. Amerik and M. Verbitsky, {\em Rational curves on hyperk\"ahler manifolds,} Int. Math. Res. Not. IMRN (2015), no. 23, 13009--13045.

\bibitem{AVsurvey} E. Amerik and M. Verbitsky, {\em Rational curves and MBM classes on hyperk\"ahler manifolds: a survey,} Rationality of varieties, 75--96, Progr. Math., 342, Birkh\"auser/Springer, Cham, 2021.



\bibitem{BHT} A. Bayer, B. Hassett, and Y. Tschinkel, {\em Mori cones of holomorphic symplectic varieties of $K3$ type,} Ann. Sci. \'Ec. Norm. Sup\'er. (4) 48 (2015), no. 4, 941--950.

\bibitem{BO} A. Bondal and D. Orlov, {\em Semiorthogonal decomposition for algebraic varieties,} arXiv:alg-geom/9506012.

\bibitem{Bridgeland} T. Bridgeland, {\em Flops and derived categories,} Invent. math. 147 (2002), no. 3, 613--632.

\bibitem{BKR} T. Bridgeland, A. King, and M. Reid, {\em The McKay correspondence as an equivalence of derived categories,} J. Amer. Math. Soc. 14 (2001), no. 3, 535--554.

\bibitem{BHL} D. Burns, Y. Hu, and T. Luo, {\em HyperK\"ahler manifolds and birational transformations in dimension $4$,} Vector bundles and representation theory (Columbia, MO, 2002), 141--149, Contemp. Math., 322, American Mathematical Society, Providence, RI, 2003.

\bibitem{Buskin} N. Buskin, {\em Every rational Hodge isometry between two $K3$ surfaces is algebraic,} J. Reine Angew. Math. 755 (2019), 127--150.

\bibitem{Ca0} A. C\u{a}ld\u{a}raru, {\em Derived categories of twisted sheaves on Calabi--Yau manifolds,} Ph.D. Thesis, Cornell University, 2000.

\bibitem{Ca} A. C\u{a}ld\u{a}raru, {\em Non-fine moduli spaces of sheaves on $K3$ surfaces,} Int. Math. Res. Not. IMRN (2002), no.~20, 1027--1056.


\bibitem{GHS} V. Gritsenko, K. Hulek, and G. K. Sankaran, {\em Moduli spaces of irreducible symplectic manifolds,} Compos. Math. 146 (2010), no. 2, 404--434.

\bibitem{DHL} D. Halpern-Leistner, {\em Derived $\Theta$-stratifications and the $D$-equivalence conjecture,} arXiv:2010.01127.

\bibitem{H1} D. Huybrechts, {\em Compact hyper-K\"ahler manifolds: basic results,} Invent. Math. 135 (1999), no. 1, 63--113.

\bibitem{H_cone} D. Huybrechts, {\em The K\"ahler cone of a compact hyperk\"ahler manifold,} Math. Ann. 326 (2003), no. 3, 499--513.

\bibitem{H2} D. Huybrechts, {\em Derived and abelian equivalence of $K3$ surfaces,} J. Algebraic Geom. 17 (2008), no. 2, 375--400.

\bibitem{H_ICM} D. Huybrechts, {\em Hyperk\"ahler manifolds and sheaves,} Proceedings of the International Congress of Mathematicians, Volume II, 450--460, Hindustan Book Agency, New Delhi, 2010.

\bibitem{HK3} D. Huybrechts, {\em Lectures on K3 surfaces,} Cambridge Stud. Adv. Math., 158, Cambridge University Press, Cambridge, 2016, xi+485 pp.




\bibitem{Polish} G. Kapustka and M. Kapustka, {\em Constructions of derived equivalent hyper-K\"ahler fourfolds,} arXiv: 2312.14543v4.

\bibitem{K} Y. Kawamata, {\em $D$-equivalence and $K$-equivalence,} J. Differential Geom. 61 (2002), no. 1, 147--171.

\bibitem{Torelli} E. Markman, {\em A survey of Torelli and monodromy results for holomorphic-symplectic varieties,} Complex and differential geometry, 257--322, Springer Proc. Math., 8, Springer, Heidelberg, 2011.


\bibitem{BB} E. Markman, {\em The Beauville--Bogomolov class as a characteristic class,} J. Algebraic Geom. 29 (2020), no.~2, 199--245.


\bibitem{Markman} E. Markman, {\em Rational Hodge isometries of hyper-K\"ahler varieties of $K3^{[n]}$-type are algebraic,} Compos. Math. 160 (2024), no. 6, 1261--1303.

\bibitem{Mon} G. Mongardi, {\em A note on the K\"ahler and Mori cones of hyperk\"ahler manifolds,} Asian J. Math. 19 (2015), no. 4, 583--591.


\bibitem{Nam} Y. Namikawa, {\em Mukai flops and derived categories,} J. Reine Angew. Math. 560 (2003), 65--76.

\bibitem{Ploog} D. Ploog, {\em Equivariant equivalences for finite group actions,} Adv. Math. 216 (2007), no. 1, 62--74.


\bibitem{Verb2} M. Verbitsky, {\em Hyperholomorphic bundles over a hyper-K\"ahler manifold,} J. Algebraic Geom. 5 (1996), no.~4, 633--669.

\bibitem{Verb} M. Verbitsky, {\em Hyperholomorphic sheaves and new examples of hyperkaehler manifolds,} Hyper-K\"ahler manifolds, by D. Kaledin, and M. Verbitsky, Mathematical Physics (Somerville), 12, International Press, Somerville, MA, 1999.

\bibitem{WW} J. Wierzba and J. Wi\'sniewski, {\em 
Small contractions of symplectic $4$-folds,} Duke Math. J. 120 (2003), no. 1, 65--95.

\bibitem{Y} K. Yoshioka, {\em Stability and the Fourier--Mukai transform. II,} Compos. Math. 145 (2009), no. 1, 112--142.

\end{thebibliography}
\end{document}